\documentclass[11pt,a4paper,twoside]{article}
\usepackage{bm, amsmath, amssymb, amsthm} 

\def\dt{\partial_t}
\def\T{T^\sharp}

\newcommand\mD{{\cal D}}

\def\<{\langle}
\def\>{\rangle}
\def\RR{\mathbb{R}}

\def\eps{\varepsilon}

\newcommand\tr{\operatorname{Tr}}
\newcommand\Div{\operatorname{div}}
\newcommand\id{\operatorname{id}}
\def\Ric{\mathcal Ric}
\def\vol{\operatorname{vol}}

\def\eq{\hspace*{-1.5mm}&=&\hspace*{-1.2mm}}
\def\plus{\hspace*{-1.5mm}&+&\hspace*{-1.2mm}}

\newtheorem{definition}{Definition}
\newtheorem{example}{Example}
\newtheorem{remark}{Remark}
\newtheorem{lemma}{Lemma}
\newtheorem{proposition}{Proposition}
\newtheorem{theorem}{Theorem}

\author{Vladimir Rovenski\footnote{Mathematical Department, University of Haifa, Mount Carmel, 31905 Haifa,  Israel
       \newline e-mail: {\tt vrovenski@univ.haifa.ac.il} } }

\title{The Einstein-Hilbert type action on almost $k$-product manifolds}

\begin{document}

\date{}

\maketitle

\begin{abstract}
A Riemannian manifold endowed with $k>2$ orthogonal complementary distributions
(called here a Riemannian almost $k$-product structure)
appears in such topics as multiply warped products, the webs composed of several foliations,
and proper Dupin hypersurfaces of real space-forms.
In~the paper we consider the mixed scalar curvature of such structure for $k>2$,
derive Euler-Lagrange equations for the Einstein-Hilbert type action with respect to adapted variations of metric,
and present them in a nice form of Einstein equation.



\end{abstract}

\section*{Introduction}

Many examples of
Riemannian metrics come (as critical points) from variational problems, a~particularly famous of which is the \textit{Einstein-Hilbert action}, given on a smooth closed manifold $M$~by
\begin{equation}\label{Eq-EH}
 J: g \to \int_M \big\{\frac1{2\mathfrak{a}}\,({\rm S}-2{\Lambda})+{\cal L}\big\}\,{\rm d}{\rm vol}_g ,
\end{equation}
e.g., \cite{besse}.
Here, $g$ is a pseudo-Riemannian metric on $M$,
S\, -- the scalar curvature of $(M,g)$,
$\Lambda$ -- a constant (the ``cosmological constant"), ${\cal L}$ -- Lagrangian describing the matter contents,
and $\mathfrak{a}=8\pi G/c^{4}$ -- the coupling constant involving the gravitational constant $G$ and the speed of light $c$.
To deal also with non-compact manifolds (``spacetimes"), it is assumed that the integral above is taken over $M$ if it converges;
otherwise, one integrates over arbitrarily large, relatively compact domain $\Omega\subset M$, which also contains supports
of variations of $g$. The Euler-Lagrange equation for \eqref{Eq-EH} (called the \textit{Einstein equation}) is
\begin{equation}\label{Eq-EC-R}
 {\rm Ric} -\,(1/2)\,{\rm S}\cdot g +\,\Lambda\,g  = \mathfrak{a}\,\Theta
\end{equation}
with the Ricci curvature ${\rm Ric}$ and the
energy-momentum tensor $\Theta$ (generalizing the stress tensor of Newtonian physics), given in a coordinates by
$\Theta_{\mu\nu}=-2\,{\partial{\cal L}}/{\partial g^{\mu\nu}} +g_{\mu\nu}{\cal L}$.
The solution of \eqref{Eq-EC-R} is a metric $g$, satisfying this equation, where the tensor $\Theta$ (describing a specified type of matter) is given.
The~classification of solutions of \eqref{Eq-EC-R} is a deep and largely unsolved problem~\cite{besse}.

 Distributions on a manifold (i.e., subbundles of the tangent bundle) appear in various situations, e.g., \cite{bf,g1967}
and are used to build up notions of integrability, and specifically of a foliated manifold.
On~a manifold equipped with an additional structure (e.g., almost product \cite{g1967} or contact),
 one can consider an analogue of \eqref{Eq-EH} adjusted to that structure.
This approach was taken in \cite{bdrs,bdr,r2018,rz-1,rz-2}, for $M$ endowed with
a distribution $\mD$ or a foliation (that can be viewed as an integrable~distribution).

In this paper, a similar change in \eqref{Eq-EH} is considered on
a connected smooth $n$-dimensional manifold endowed with $k\ge2$ pairwise orthogonal $n_i$-dimensional distributions with $\sum n_i = n$.
The notion of a {multiply warped product}, e.g., \cite{chen1}, is a special case of
this structure, which can be also viewed in the theory of webs composed of foliations of different dimensions, see~\cite{AG2000}.
The~\textit{mixed Einstein-Hilbert action} on $(M,\mD_1,\ldots,\mD_k)$,
\begin{equation}\label{Eq-Smix}
 J_{\mD}: g\mapsto\int_{M} \Big\{\frac1{2\mathfrak{a}}\,({\rm S}_{\,\mD_1,\ldots,\mD_k}-2\,{\Lambda})+{\cal L}\Big\}\,{\rm d}\vol_g,
\end{equation}
is an analog of \eqref{Eq-EH}, where ${\rm S}\,$ is replaced by the mixed scalar curvature ${\rm S}_{\,\mD_1,\ldots,\mD_k}$, see~\eqref{E-Smix-k}.
The~geometrical part of \eqref{Eq-Smix} is $J^g_{\mD}: g\mapsto\int_{M} {\rm S}_{\,\mD_1,\ldots,\mD_k}\,{\rm d}\vol_g$ -- the total mixed scalar curvature.
 The~mixed scalar curvature is the simplest curvature invariant of a pseudo-Riemannian an almost product structure, which can be defined as an averaged sum of sectional curvatures of planes that non-trivially intersect with both of the distributions.
Investigation of ${\rm S}_{\,\mD_1,\ldots,\mD_k}$ can lead to multiple results regarding the existence of foliations and submersions with interesting geometry, e.g., integral formulas and splitting results, curvature prescribing and variational problems, see \cite{rov-5} for $k=2$.
Varying \eqref{Eq-Smix} as a functional of adapted metric $g$, we obtain the Euler-Lagrange equations in the form of Einstein equation \eqref{Eq-EC-R}, i.e.,
\begin{equation}\label{E-gravity}
 \Ric_{\,{\mD}} -\,(1/2)\,{\cal S}_{\,\mD}\cdot g +\,\Lambda\,g = \mathfrak{a}\,\Theta
\end{equation}
(for $k=2$, see \cite{bdrs,rz-1,rz-2}),
where the Ricci tensor and the scalar curvature are replaced by
the Ricci type tensor $\Ric_{\,{\mD}}$, see \eqref{E-main-0ij}, and its~trace ${\cal S}_{\,\mD}$.
 By the equality
 \[
  {\rm S}={\rm S}_{\,\mD_1,\ldots,\mD_k}+\sum\nolimits_{\,i} {\rm S}({\mD_i}),
 \]
where ${\rm S}({\mD_i})$ is the scalar curvature along the distribution $\mD_i$,
one can combine the action \eqref{Eq-EH} with \eqref{Eq-Smix} to obtain the \textit{perturbed Einstein-Hilbert action} on $(M,\mD_1,\ldots,\mD_k)$:
\[
 \bar J_{\eps}: g\mapsto\int_{M} \big\{\frac1{2\mathfrak{a}}\,({\rm S}+\eps\,{\rm S}_{\,\mD_1,\ldots,\mD_k} -2\,{\Lambda})
+{\cal L}\big\}\,{\rm d}\vol_g
\]
with $\eps\in\RR$, whose critical points describe the ``space-times" in extended theory of gravity.

In~the paper we consider the geometrical part of the Einstein-Hilbert type action \eqref{Eq-Smix} (called the total mixed scalar curvature),
derive Euler-Lagrange equations with respect to adapted variations of metric (which generalize equations for $k=2$ in \cite{rz-1}),
and present them in a nice form \eqref{E-gravity} of Einstein equation.
 We delegate the following for further study:
a)~generalize our results for more general variations of metrics;
b)~extend our results for variations of connections (as in Einstein-Cartan theory);
c)~find more applications of our results in geometry and physics.

\section{Preliminaries}

Let $M$ be a smooth connected $n$-dimensional manifold with the Levi-Civita connection $\nabla$ and the curvature tensor $R$.
A~{pseudo-Riemannian metric} $g=\<\cdot,\cdot\>$ of index $q$ on $M$ is an element $g\in{\rm Sym}^2(M)$
(of the space of symmetric $(0,2)$-tensors)
such that each $g_x\ (x\in M)$ is a {non-degenerate bilinear form of index} $q$ on the tangent space $T_xM$.
For~$q=0$ (i.e., $g_x$ is positive definite) $g$ is a Riemannian metric and for $q=1$ it is called a Lorentz metric.
Let~${\rm Riem}(M)\subset{\rm Sym}^2(M)$ be the subspace of pseudo-Riemannian metrics of a given signature.

A distribution $\mD$ on $M$ is \textit{non-degenerate},
if $g_x$ is non-degenerate on $\mD_x\subset T_x M$ for all $x\in M$; in this case, the orthogonal complement
of~${\mD}^\bot$ is also non-degenerate.

Let $M$ be endowed with $k\ge2$ pairwise transversal $n_i$-dimensional distributions ${\mD}_i\ (1\le i\le k)$ with $\sum n_i = n$.
Denote by ${\rm Riem}(M,\mD_1,\ldots\mD_k)\subset {\rm Riem}(M)$ the subspace of pseudo-Riemannian metrics making $\{\mD_i\}$
pairwise orthogonal and non-degenerate.
A $(M;{\mD}_1,\ldots,{\mD}_k)$ with a compatible metric $g$ will be called here a \textit{Riemannian almost $k$-product manifold}, see
\cite{r-IF-k}.
Let $P_i:TM\to{\mD}_i$ be the orthoprojector, and $P_{i}^\bot=\id_{\,TM}-P_i$ be the orthoprojector onto ${\mD}_{i}^\bot$.
The~second fundamental form $h_i:{\mD}_i\times {\mD}_i\to {\mD}_i^\bot$ (symmetric)
and the integrability tensor $T_i:{\mD}_i\times {\mD}_i\to {\mD}_i^\bot$ (skew-symmetric)
of ${\mD}_i$ are defined by
\[
 h_i(X,Y) = \frac12\,P_{i}^\bot(\nabla_XY+\nabla_YX),\quad
 T_i(X,Y) = \frac12\,P_{i}^\bot(\nabla_XY-\nabla_YX)=\frac12\,P_{i}^\bot\,[X,Y].
\]
Similarly, $h_{i}^\bot,\,H_{i}^\bot=\tr_g h_{i}^\bot,\,T_{i}^\bot$ are
the~second fundamental forms, mean curvature vector fields and the integrability tensors
of distributions ${\mD}_{i}^\bot$ in $M$.
 Note that
 $H_i=\sum\nolimits_{\,j\ne i} P_j H_i$,
etc.
Recall that a distribution ${\mD}_i$ is called integrable if $T_i=0$,
and ${\mD}_i$ is called {totally umbilical}, {harmonic}, or {totally geodesic},
if ${h}_i=({H}_i/n_i)\,g,\ {H}_i =0$, or ${h}_i=0$, respectively.

The ``musical" isomorphisms $\sharp$ and $\flat$ will be used for rank one and symmetric rank 2 tensors.
For~example, if $\omega \in\Lambda^1(M)$ is a 1-form and $X,Y\in {\mathfrak X}_M$ then
$\omega(Y)=\<\omega^\sharp,Y\>$ and $X^\flat(Y) =\<X,Y\>$.
For arbitrary (0,2)-tensors $B$ and $C$ we also have $\<B, C\> =\tr_g(B^\sharp C^\sharp)=\<B^\sharp, C^\sharp\>$.

The shape operator $(A_i)_Z$ of ${\mD}_i$ with $Z\in{\mD}_i^\bot$, and the operator $(T_i)^\sharp_{Z}$ are defined~by
\[
 \<(A_i)_Z(X),Y\>= \,h_i(X,Y),Z\>,\quad \<(T_i)_Z^\sharp(X),Y\>=\<T_i(X,Y),Z\>, \quad X,Y \in {\mD}_i .
\]
Given $g\in{\rm Riem}(M,\mD_1,\ldots\mD_k)$, there exists a~local $g$-orthonormal frame $\{E_1,\ldots,E_n\}$ on $M$, where
 $\{E_1,\ldots, E_{n_1}\}\subset{{\mD}_1}$ and
 $\{E_{n_{i-1}+1},\ldots, E_{n_i}\}\subset{{\mD}_i}$ for $2\le i\le k$, and $\eps_a=\<E_{a},E_{a}\>\in\{-1,1\}$.
All quantities defined below using an adapted frame do not depend on the choice of this~frame.
 The squares of norms of tensors are obtained using
\begin{equation*}
 \<h_i,h_i\>=\hskip-2.7mm\sum\limits_{n_{i-1}<a,b\,\le n_i}\!\!\!\!\!\eps_a\eps_b\,\<h_i({E}_a,{E}_b),h_i({E}_a,{E}_b)\>, \quad
 \<T_i,T_i\>=\hskip-2.7mm\sum\limits_{n_{i-1}<a,b\,\le n_i}\!\!\!\!\!\eps_a\eps_b\,\<T_i({E}_a,{E}_b),T_i({E}_a,{E}_b)\>.
\end{equation*}
The \textit{divergence} of a vector field $X\in\mathfrak{X}_M$ is given by
 $(\Div X)\,{\rm d}\vol_g = {\cal L}_{X}({\rm d}\vol_g)$,
where ${\rm d} \vol_g$ is the volume form of $g$. One may show that
$\Div X=\tr(\nabla X)=\Div_i X+\Div_i^\bot X$, where
\[
 \Div_i X =\sum\nolimits_{n_{i-1}<a\,\le n_i}\eps_a\,\<\nabla_{a}\,X, {E}_a\>,\quad
 \Div_i^\bot X=\sum\nolimits_{b\ne(n_{i-1}, n_i]}\eps_b\,\<\nabla_{b}\,X, E_b\> .
\]
Observe that for $X\in\mD_i$ we have
\begin{equation}\label{E-divN}
 {\Div}_i^\bot X = \Div X +\<X,\,H_i^\bot\>.
\end{equation}
For a $(1,2)$-tensor $Q$ define a $(0,2)$-tensor ${\Div}_i^\bot Q$ by
\[
 ({\Div}_i^\bot\,Q)(X,Y) = \sum\nolimits_{b\ne(n_{i-1}, n_i]}\eps_b\,\<(\nabla_{b}\,Q)(X,Y), E_b\>,\quad X,Y \in \mathfrak{X}_M.
\]
For a~${\mD}_i$-valued $(1,2)$-tensor $Q$, similarly to \eqref{E-divN}, we have
${\Div}_i^\bot\,Q = \Div Q+\<Q,\,H_i\>$, where
\begin{equation*}
 ({\Div}_i\,Q)(X,Y) =\sum\nolimits_{n_{i-1}<a\,\le n_i}\eps_a\,\<(\nabla_{a}\,Q)(X,Y), E_a\> = -\<Q(X,Y), H_i\>,
\end{equation*}
and $\<Q,\,H_i\>(X,Y)=\<Q(X,Y),\,H_i\>$ is a $(0,2)$-tensor.
For example, $\Div_i^\bot h_i = \Div h_i+\<h_i,\,H_i\>$.
For~a~function $f$ on $M$, we use the notation $P_i^\bot(\nabla f)$ for the projection of $\nabla f$ onto $\mD_i^\bot$.

The ${\mD}_i$-\textit{deformation tensor} ${\rm Def}_{\mD_i}Z$ of $Z\in\mathfrak{X}_M$
is the symmetric part of $\nabla Z$ restricted to~${\mD}_i$,
\begin{equation*}
 2\,{\rm Def}_{\mD_i}\,Z(X,Y)=\<\nabla_X Z, Y\> +\<\nabla_Y Z, X\>,\quad X,Y\in \mD_i.
\end{equation*}
The Casorati type operators ${\cal A}_i:\mD_i\to\mD_i$ and ${\cal T}_i:\mD_i\to\mD_i$ and the symmetric $(0,2)$-tensor $\Psi_i$, see \cite{bdr,rz-1}, 
are defined using $A_i$ and $T_i$ by
\begin{eqnarray*}
 && {\cal A}_i=\sum\nolimits_{\,E_a\in\mD_i^\bot}\eps_a ((A_i)_{E_a})^2,\quad
 {\cal T}_i=\sum\nolimits_{\,E_a\in\mD_i^\bot}\eps_a((T_i)_{E_a}^\sharp)^2,\\
 && \Psi_i(X,Y) = \tr((A_i)_Y (A_i)_X+(T_i)^\sharp_Y (T_i)^\sharp_X), \quad X,Y\in\mD_i^\bot.
\end{eqnarray*}
We define a self-adjoint $(1,1)$-tensor ${\cal K}_i$ by the formula with Lie bracket,
\[
 {\cal K}_i = \sum\nolimits_{\,E_a\in\mD_i^\bot} \eps_{\,a}\,[(\T_i)_{E_a}, (A_i)_{E_a}]
 = \sum\nolimits_{\,E_a\in\mD_i^\bot} \eps_a \big((T_i)_{E_a} (A_i)_{E_a} - (A_i)_{E_a} (T_i)_{E_a}\big).
\]
For any $(1,2)$-tensors $Q_1,Q_2$ and a $(0,2)$-tensor $S$ 
on $TM$, define the following $(0,2)$-tensor $\Upsilon_{Q_1,Q_2}$:
\[
 \<\Upsilon_{Q_1,Q_2}, S\> =  \sum\nolimits_{\,\lambda, \mu} \eps_\lambda\, \eps_\mu\,
 \big[S(Q_1(e_{\lambda}, e_{\mu}), Q_2( e_{\lambda}, e_{\mu})) + S(Q_2(e_{\lambda}, e_{\mu}), Q_1( e_{\lambda}, e_{\mu}))\big],
\]
where on the left-hand side we have the inner product of $(0,2)$-tensors induced by $g$,
$\{e_{\lambda}\}$ is a local orthonormal basis of $TM$ and $\eps_\lambda = \<e_{\lambda}, e_{\lambda}\>\in\{-1,1\}$.

\begin{remark}\rm
 If $g$ is definite then $\Upsilon_{\,h_i,h_i}=0$ if and only if $h_i=0$.
Indeed, for any $X\in{\cal D}_i$ we~have
\begin{equation*}
 \<\Upsilon_{\,h_i,h_i}, X^\flat\otimes X^\flat\> = 2\sum\nolimits_{\,a,b} \<X, h_i(E_{a}, E_{b})\>^{2}.
\end{equation*}
The above sum is equal to zero if and only if every summand vanishes. This yields $h_i=0$.
Thus, $\Upsilon_{\,h_i,h_i}$ is a ``measure of non-total geodesy" of ${\cal D}_i^\bot$.
Similarly, if $\Upsilon_{\,T_i,T_i}=0$ then
\[
 \<\Upsilon_{\,T_i,T_i}, X^\flat\otimes X^\flat\> = 2\sum\nolimits_{\,a,b} \<X, T_i(E_{a},E_{b})\>^{2}.
\]
Hence, if $g$ is definite then the condition $\Upsilon_{\,T_i,T_i}=0$ is equivalent to $T_i=0$.
Therefore, $\Upsilon_{\,T_i,T_i}$ can be viewed as a ``measure of non-integrability" of ${\cal D}_i$.
\end{remark}

\section{The mixed scalar curvature}

A plane in $TM$ spanned by two vectors belonging to different distributions ${\mD}_i$ and ${\mD}_j$ will be called~\textit{mixed},
and the its sectional curvature is called mixed.
Similarly to the case of $k=2$, the mixed scalar curvature of $(M,g;\mD_1,\ldots,\mD_k)$ is defined as an averaged mixed sectional curvature.

\begin{definition}[see \cite{r-IF-k}]\rm
Given $g\in{\rm Riem}(M;\mD_1,\ldots,\mD_k)$ with $k\ge2$,
the following function on $M$ will be called the \textit{mixed scalar curvature}:
\begin{equation}\label{E-Smix-k}
 {\rm S}_{\,{\mD}_1,\ldots,{\mD}_k}=\sum\nolimits_{\,i<j}{\rm S}({\mD}_i,{\mD}_j),
\end{equation}
where
\[
 {\rm S}({\mD}_i,{\mD}_j) = \sum\nolimits_{\,n_{i-1}<a\,\le n_i,\ n_{j-1}<b\le n_j}
 \eps_a\,\eps_b\<R(E_a,{E}_b)\,E_a,\,{E}_{b}\>,\quad i\ne j.
\]
The~following symmetric $(0,2)$-tensor $r$ will be called the \textit{partial Ricci tensor}:
\[
  r(X,Y) = \frac12\sum\nolimits_{\,i=1}^k {r}_{i}(X,Y),
\]
where the {partial Ricci tensor} related to $\mD_i$ is
\begin{equation}\label{E-Rictop2}
 {r}_{i}(X,Y) = \sum\nolimits_{n_{i-1}<a\,\le n_i} \eps_a\, \<R_{\,E_a,\,P_i^\bot\,X}\,E_a, \, P_i^\bot\,Y\>, \quad X,Y\in \mathfrak{X}_M.
\end{equation}
\end{definition}

\begin{proposition} We have
\begin{equation*}
  {\rm S}_{\,{\mD}_1,\ldots,{\mD}_k} = \frac12\sum\nolimits_{\,i}{\rm S}_{\,{\mD}_i,{\mD}^\bot_i}
  =\tr_g\, r.
\end{equation*}
\end{proposition}

\begin{proof} This directly follows from definitions \eqref{E-Smix-k} and \eqref{E-Rictop2} and equality
$\tr_g\, r_i = {\rm S}_{\,{\mD}_i,{\mD}^\bot_i}$.
\end{proof}

The~following formula for a Riemannian manifold $(M,g)$ endowed with two complementary ortho\-gonal distributions ${\mD}$ and ${\mD}^\bot$,
see \cite{wa1}:
\begin{equation}\label{E-PW}
 \Div(H+H^\bot) ={\rm S}_{\,{\mD},{\mD}^\bot}+\<h,h\>+\<h^\bot,h^\bot\>
 -\<H,H\>-\<H^\bot,H^\bot\>-\<T,T\>-\<T^\bot,T^\bot\> ,
\end{equation}
has many interesting global corollaries (e.g., decomposition criteria using the sign of ${\rm S}$, \cite{step1}).
In \cite{r-IF-k}, we generalized \eqref{E-PW} to $(M,g)$ with $k>2$ distributions
and gave applications to splitting and isometric immersions of manifolds., in particular, multiply warped products.

The mixed scalar curvature of a pair of distributions $(\mD_i,\mD_i^\bot)$ on $(M,g)$ is
\begin{equation}\label{eq-wal2}
 {\rm S}_{\,\mD_i,\mD_i^\bot} = \sum\nolimits_{\,n_{i-1}<a\,\le n_i,\ b\ne(n_{i-1}, n_i]}\eps_a \eps_b\,\<R_{\,E_a, E_b} E_a, E_b\>.
\end{equation}
 If~$\mD_i$ is spanned by a unit vector field $N$, i.e., $\<N,N\>=\eps_N\in\{-1,1\}$,
then ${\rm S}_{\,\mD_i,\mD_i^\bot}=\eps_N{\rm Ric}_{N,N}$, where ${\rm Ric}_{N,N}$ is the Ricci curvature in the $N$-direction.
We have ${\rm S}_{\,{\mD}_i,{\mD}_i^\bot}=\tr_g r_i=\tr_g r_{i}^\bot$.
If $\dim\mD_i=1$ then $r_i=\eps_N\,R_N$, where $R_N=R(N,\,\cdot)\,N$ is the Jacobi operator,
and
$r_i^\bot={\rm Ric}_{N,N}\,g_i^\bot$, where the symmetric (0,2)-tensor $g_i^\perp$ is defined by
$g_i^\perp(X,Y)=\<P_i^\bot X, P_i^\bot Y\>$ for $X,Y \in \mathfrak{X}_M$.
 The~following presentation of the partial Ricci tensor in \eqref{E-Rictop2} is valid, see \cite{bdr,rz-1}:
\begin{equation}\label{E-genRicN}
 r_{i} = \Div h_i +\<h_i,\,H_i\>-{\cal A}_i^\flat-{\cal T}_i^\flat-\Psi_i^\bot+{\rm Def}_{\mD_i^\bot}\,H_i^\bot.
\end{equation}
Tracing \eqref{E-genRicN} over ${\cal D}_i$ and applying the equalities
\begin{eqnarray*}
 && \tr_{g}\,({\Div}\,h_i) =\Div H_i,\quad
 \tr\<h_i,\,H_i\> = \<H_i,H_i\>,\quad
 \tr_{g}\Psi_i^\bot =\<h_i^\bot,h_i^\bot\> - \<T_i^\bot,T_i^\bot\>,\\
 && \tr{\cal A}_i = \<h_i,h_i\>,\quad
 \tr{\cal T}_i = -\<T_i,T_i\>,\quad
 \tr_{g}\,({\rm Def}_{{\cal D}_i^\bot}\,H_i^\bot) = \Div H_i^\bot +g(H_i^\bot, H_i^\bot),
\end{eqnarray*}
we get \eqref{E-PW} with $\mD=\mD_i$.

\begin{theorem}[\cite{r-IF-k}]
For a Riemannian almost $k$-product manifold $(M,g;{\mD}_1,\ldots,{\mD}_k)$ we have
\begin{eqnarray}\label{E-PW3-k}
\nonumber
 && \Div \sum\nolimits_{\,i} (H_{i}+H^\bot_{i}) = 2\,{\rm S}_{\,{\mD}_1,\ldots,{\mD}_k} \\
 && +\sum\nolimits_{\,i} \big(\<h_{i},h_{i}\>-\<H_{i},H_{i}\>-\<T_{i},T_{i}\>
 +\<h^\bot_{i},h^\bot_{i}\>-\<H^\bot_{i},H^\bot_{i}\>-\<T^\bot_{i},T^\bot_{i}\>\big).
\end{eqnarray}
\end{theorem}

\begin{example}\rm
To illustrate the proof of \eqref{E-PW3-k} for $k>2$, consider the case of $k=3$.
Using \eqref{E-PW} for two distributions, ${\mD}_1$ and ${\mD}_{1}^\bot={\mD}_2\oplus{\mD}_3$, according to \eqref{E-PW} we get
\begin{eqnarray}\label{E-PW3-1}
\nonumber
 \Div(H_1 + H_{1}^\bot) \eq 2\,{\rm S}_{\,{\mD}_1, {\mD}_{1}^\bot} + (\<h_1,h_1\> - \<H_1,H_1\> - \<T_1,T_1\>) \\
 \plus (\<h_{1}^\bot,h_{1}^\bot\> -\<H_{1}^\bot,H_{1}^\bot\> - \<T_{1}^\bot,T_{1}^\bot\>),
\end{eqnarray}
and similarly for $({\mD}_2,\,{\mD}_{2}^\bot)$ and $({\mD}_3,\,{\mD}_{3}^\bot)$.
Summing 3 copies of \eqref{E-PW3-1}, we obtain \eqref{E-PW3-k} for $k=3$:
\begin{eqnarray*}
 && \Div\sum\nolimits_{\,i}\big( H_i+ H_{i}^\bot\big) = 2\,{\rm S}_{\,{\mD}_1,{\mD}_2,{\mD}_3} \\
 && +\sum\nolimits_{\,i}\big(\<h_i,h_i\> -\<H_i,H_i\> - \<T_i,T_i\>
 +\<h_{i}^\bot,h_{i}^\bot\> - \<H_{i}^\bot,H_{i}^\bot\> -\<T_{i}^\bot,T_{i}^\bot\>\big) .
\end{eqnarray*}
\end{example}

\begin{remark}\rm Using Stokes' Theorem for \eqref{E-PW3-k} on a closed manifold $(M,g;{\mD}_1,\ldots,{\mD}_k)$ yields the integral formula for any $k\in\{2,\ldots,n\}$,
which for $k=2$ directly follows from \eqref{E-PW},
\begin{eqnarray*}
 \int_M\big( 2\,{\rm S}_{\,{\cal D}_1,\ldots,{\cal D}_k}
 +\sum\nolimits_{\,i}(\<h_{i},h_{i}\>-\<H_{i},H_{i}\>-\<T_{i},T_{i}\> \\
 +\<h^\bot_{i},h^\bot_{i}\>-\<H^\bot_{i},H^\bot_{i}\>-\<T^\bot_{i},T^\bot_{i}\>)\big)\,{\rm d}\vol_g =0.
\end{eqnarray*}
\end{remark}

\section{Adapted variations of metric}

We consider smooth $1$-parameter variations $\{g_t\in{\rm Riem}(M):\,|t|<\eps\}$ of the metric $g_0 = g$.
Let the infinitesimal variations, represented by a symmetric $(0,2)$-tensor
\[
 {B}(t)\equiv\partial g_t/\partial t,
\]
be supported in a relatively compact domain $\Omega$ in $M$,
i.e., $g_t=g$ and ${B}_t=0$ outside $\Omega$ for $|t|<\eps$.
A variation $g_t$ is \emph{volume-preserving} if ${\rm Vol}(\Omega,g_t) = {\rm Vol}(\Omega,g)$ for all $t$.
 We~adopt the notations $\partial_t \equiv \partial/\partial t,\ {B}\equiv{\dt g_t}_{\,|\,t=0}=\dot g$,
but we shall also write $B$ instead of $B_t$ to make formulas easier to read, wherever it does not lead to confusion.
Since $B$ is symmetric, then $\<C,\,B\>=\<{\rm Sym}(C),\,B\>$ for any  $(0,2)$-tensor $C$.
Denote by $\otimes$ the product of tensors.

\begin{definition}
\rm
A family of metrics
\begin{equation*}
 \{g(t)\in{\rm Riem}(M,\mD_1,\ldots\mD_k):\, |t|<\eps\}
\end{equation*}
such that $g_0=g$ will be called an \textit{adapted variation}.
In other words, $\mD_i$ and $\mD_j$ are $g_t$-orthogonal for all $i\ne j$ and~$t$.
 An adapted variation $g_t$ will be called a ${\cal D}_j$-\textit{variation} (for some $j\in[1,k]$) if
\[
 g_t(X,Y)=g_0(X,Y),\quad X,Y\in\mD_j^\bot,\quad |t|<\eps.
\]
\end{definition}

For an adapted variation we have $g_t=\bigoplus_{\,j=1}^k g_j(t)$, where $g_j(t) =g_t|_{\,\mD_j}$.
Thus, the tensor $B_t=\dt\, g_t$ of an adapted variation of metric $g$ on $(M;\mD_1,\ldots,\mD_k)$ can be decomposed into the sum of derivatives of ${\cal D}_j$-variations; namely, $B_t=\sum_{\,j=1}^k {B}_j(t)$, where ${B}_j(t) =\dt\,g_j(t) =B_t|_{\,\mD_j}$.

\begin{lemma}
Let a local adapted frame $\{E_a\}$ evolve by $g_t\in{\rm Riem}(M,\mD_1,\ldots\mD_k)$ according to
\begin{equation*}
 \dt E_a=-(1/2)\,{B}^\sharp_t(E_a).
\end{equation*}
 Then, $\{E_a(t)\}$ is a $g_t$-orthonormal adapted frame for all $\,t$.
\end{lemma}

\begin{proof}
For $\{E_a(t)\}$ we have
\begin{eqnarray*}
 &&\dt(g_t(E_a, E_b)) = g_t(\dt E_a(t), E_b(t)) +g_t(E_a(t), \dt E_b(t))
 +(\dt g_t)(E_a(t), E_b(t))  \\
 &&= {B}_t(E_a(t), E_b(t))-\frac12\,g_t({B}_t^\sharp(E_a(t)), E_b(t)) -\frac12\,g_t(E_a(t), {B}_t^\sharp(E_b(t)))=0.
\end{eqnarray*}
From this the claim follows.
\end{proof}

The following proposition was proved in \cite{rz-2} when $k=2$.

\begin{proposition}\label{propvar1}
If $g_t$ is a ${\cal D}_j$-variation of $g$ on $(M;\mD_1,\ldots,\mD_k)$, then
\begin{eqnarray*}
 &&\hskip-9mm
 \dt\<{h}_j^\bot, {h}_j^\bot\> =  -\<(1/2)\Upsilon_{h_j^\bot,h_j^\bot},\ B_j\>,\\
 &&\hskip-9mm \dt \<h_j,\ h_j\> = \<\,\Div{h}_j + {\cal K}_j^\flat,\ B_j\> - \Div\<h_j, B_j\>,\\
 &&\hskip-9mm \dt g({H}_j^\bot, {H}_j^\bot) = -\<\,({H}_j^\bot)^\flat\otimes({H}_j^\bot)^\flat,\ B_j\>,\\
 &&\hskip-9mm \dt g(H_j, H_j) = \<\,(\Div H_j)\,g_j,\ B_j\> -\Div((\tr_{\,{\mD}_j} B_j^\sharp) H_j), \\
 &&\hskip-9mm \dt\<{T}_j^\bot, {T}_j^\bot\> = \<\,(1/2)\Upsilon_{T_j^\bot,T_j^\bot},\ B_j\>,\\
 &&\hskip-9mm \dt\<T_j,\ T_j\> = \< 2\,{\cal T}_j^\flat,\ B_j\> ,
\end{eqnarray*}
and for $i\ne j$ (when $k>2$) we have dual equations
\begin{eqnarray*}
 &&\hskip-9mm \dt \<h_i^\bot,\ h_i^\bot\> = \<\,\Div{h}_i^\bot + ({\cal K}_i^\bot)^\flat,\ B_j\> - \Div\<h_i^\bot, B_j\>,\\
 &&\hskip-9mm
 \dt\<{h}_i, {h}_i\> =  \<-(1/2)\Upsilon_{h_i,h_i},\ B_j\>,\\
 &&\hskip-9mm \dt g(H_i^\bot, H_i^\bot) = \<\,(\Div H_i^\bot)\,g_j,\ B_j\>
 -\Div((\tr_{\,{\mD}_i^\bot} (B_j)^\sharp) H_i^\bot), \\
 &&\hskip-9mm \dt g({H}_i, {H}_i) = -\<\,{H}_i^\flat\otimes{H}_i^\flat,\ B_j\>,\\
 &&\hskip-9mm \dt\<T_i^\bot,\ T_i^\bot\> = \< 2\,({\cal T}_i^\bot)^\flat,\ B_j\> ,\\
 &&\hskip-9mm \dt\<{T}_i, {T}_i\> = \<\,(1/2)\Upsilon_{T_i,T_i},\ B_j\>.
\end{eqnarray*}
\end{proposition}

 For any variation $g_t$ of metric $g$ on $M$ with $B=\dt g$ we have
\begin{equation}\label{E-dotvolg}
 \partial_t\,\big({\rm d}\vol_{g}\!\big) = \frac12\,(\tr_{g} B)\,{\rm d}\vol_{g}.
\end{equation}
Differentiating the well-known formula
 $\Div X \cdot {\rm d}\vol_g = {\cal L}_{X}({\rm d}\vol_g)$
and using \eqref{E-dotvolg}, we obtain
\begin{equation}\label{dtdiv}
 \dt\,(\Div X) = \Div (\dt X) +\frac{1}{2}\,X(\tr_{g} B)
\end{equation}
for any variation $g_t$ of metric and a $t$-dependent vector field $X$ on $M$.
By \eqref{dtdiv} and \eqref{E-dotvolg}, using the Divergence Theorem, we have
\begin{equation}\label{E-DivThm-2}
 \frac{d}{dt}\int_M (\Div X)\,{\rm d}\vol_g  =\int_M \Div\big(\dt X+\frac12\,(\tr_g B) X\big)\,{\rm d}\vol_g = 0
\end{equation}
for any variation $g_t$ with ${\rm supp}\,(\dt g)\subset\Omega$,
and $t$-dependent $X\in\mathfrak{X}_M$ with ${\rm supp}\,(\dt X)\subset\Omega$.

The following theorem allows us to restore the partial Ricci curvature \eqref{E-main-0ij}.
It is based on calculating the variations with respect to $g$ of components in \eqref{E-PW}
and using \eqref{E-DivThm-2} for divergence~terms.
By this theorem and Definition~\ref{D-Ric-D} we conclude that
an adapted metric $g$ is critical for the action \eqref{Eq-Smix}
with respect to volume-preserving adapted variations of metric if and only if \eqref{E-gravity} holds.

\begin{theorem}[see \cite{rz-2}]
A metric $g\in{\rm Riem}(M,{\mD}_1,\ldots,\mD_k)$ is critical for the geometrical part of \eqref{Eq-Smix},
i.e., $\Lambda=0={\cal L}$, with respect to volume-preserving adapted variations if and only if
\begin{eqnarray}\label{ElmixDDvp}
\nonumber
 && \Div{h_j} +{\cal K}_j^\flat -\frac12\Upsilon_{h_j^\bot,h_j^\bot}
 +({H}_j^\bot)^\flat\otimes({H}_j^\bot)^\flat -\frac12\Upsilon_{T_j^\bot,T_j^\bot} -2\,{\cal T}_j^\flat
 +\sum\nolimits_{\,i\ne j}\big(\Div{h_i^\bot}|_{\mD_j}\\
\nonumber
 && +\,(P_j{\cal K}_i^\bot)^\flat +(P_j{H}_i)^\flat\otimes(P_j{H}_i)^\flat-\frac12\Upsilon_{P_jh_i,P_jh_i}
 -\frac12\Upsilon_{P_jT_i,P_jT_i} -2\,(P_j{\cal T}_i^\bot)^\flat\big) \\
 && =\,\big(\,{\rm S}_{\,{\mD}_1,\ldots,\mD_k} - \Div(H_j + \sum\nolimits_{\,i\ne j} {H}_i^\bot) + \lambda_j\big)\,g_j
\end{eqnarray}
for some $\lambda_j\in\RR$ and $1\le j\le k$.
\end{theorem}

\begin{proof} 
Let $g_t$ be a ${\cal D}_j$-variation for some $j$, and let
\[
 Q(g):=2\,{\rm S}_{\,{\mD}_1,\ldots,\mD_k} -\Div\sum\nolimits_{\,i}(H_i+H_i^\bot).
\]
Then
\begin{equation*}
 {\rm\frac{d}{dt}}\,J_{{\cal D},\Omega}(g_t)_{\,|\,t=0} ={\rm\frac{d}{dt}}\int_{\Omega} Q(g_t)\,{\rm d}\vol_{g_t\,|\,t=0}
 +{\rm\frac{d}{dt}}\int_{\Omega}\Div\sum\nolimits_{\,i}(H_i+H_i^\bot)\,{\rm d}\vol_{g_t\,|\,t=0}.
\end{equation*}
For adapted variations of $g$ supported in $\Omega$, both fields $\dt\sum\nolimits_{\,i}(H_i+H_i^\bot)$ and
$(\tr {B}^\sharp)\sum\nolimits_{\,i}(H_i+H_i^\bot)$ vanish on $\partial \Omega$,
and by \eqref{dtdiv} we get ${\rm\frac{d}{dt}}\int_{\Omega}\Div\sum\nolimits_{\,i}(H_i+H_i^\bot)\,{\rm d}\vol_g =0$.
Thus, we have
\begin{equation*}
 {\rm\frac{d}{dt}}\,J_{{\cal D},\Omega}(g_t)_{\,|\,t=0} ={\rm\frac{d}{dt}}\int_{\Omega} Q(g_t)\,{\rm d}\vol_{g_t \,|\,t=0},
\end{equation*}
and $Q(g)$ can be presented using \eqref{E-PW} and \eqref{E-PW3-k} as
\begin{equation*}
 Q(g) = \frac12\sum\nolimits_{\,i}\big(\<T_{i},T_{i}\> -\<h_{i},h_{i}\> +\<H_{i},H_{i}\> +\<T^\bot_{i},T^\bot_{i}\> -\<h^\bot_{i},h^\bot_{i}\> +\<H^\bot_{i},H^\bot_{i}\>\big)\,.
\end{equation*}
By Proposition~\ref{propvar1}, we get
\begin{eqnarray*}
 &&
 \dt\big(-\<{h}_j^\bot, {h}_j^\bot\>-\<h_j,\ h_j\> +g({H}_j^\bot, {H}_j^\bot)+g(H_j, H_j)+\<{T}_j^\bot,{T}_j^\bot\>
 +\<T_j, T_j\>\big) \\
 && =\<\,-\Div{h}_j -{\cal K}_j^\flat -({H}_j^\bot)^\flat\otimes({H}_j^\bot)^\flat +(1/2)\Upsilon_{h_j^\bot,h_j^\bot} +(1/2)\Upsilon_{T_j^\bot,T_j^\bot}
 + 2\,{\cal T}_j^\flat  \\
 && +\, (\Div H_j)\,g_j,\ B_j\> + \Div\big(\<h_j, B_j\> - (\tr_{\,{\mD}_j} B_j^\sharp) H_j\big),
\end{eqnarray*}
and for $i\ne j$ (when $k>2$) we have
\begin{eqnarray*}
 && \dt\big( -\<h_i^\bot,\ h_i^\bot\> -\<{h}_i, {h}_i\>
 +g(H_i^\bot, H_i^\bot) +g(H_i, H_i) +\<T_i^\bot, T_i^\bot\> +\<T_i, T_i\> \big) \\
 && = \<\,-\Div{h}_i^\bot - ({\cal K}_i^\bot)^\flat -{H}_i^\flat\otimes{H}_i^\flat
  +(1/2)\Upsilon_{h_i,h_i} +(1/2)\Upsilon_{T_i,T_i} + 2\,({\cal T}_i^\bot)^\flat \\
 && +\,(\Div H_i^\bot)\,g_j,\ B_j\> +\Div(\<h_i^\bot, B_j\> -(\tr_{\,{\mD}_i^\bot}(B_j)^\sharp) H_i^\bot).
\end{eqnarray*}
We use the above to derive $\dt Q(g)$.
Removing integrals of divergences of vector fields compactly supported in $\Omega$, we get
\begin{eqnarray}\label{E-Sc-var1a}
\nonumber
&& \int_{\Omega}\dt Q(g_t)_{\,|\,t=0}\,{\rm d}\vol_g
 = \frac 12\int_{\Omega}\Big(\sum\nolimits_{\,i\ne j}\<-\Div{h_i^\bot}-({\cal K}_i^\bot)^\flat -{H}_i^\flat\otimes{H}_i^\flat +\frac12\,\Upsilon_{h_i,h_i}\\
\nonumber
&& +\,\frac12\,\Upsilon_{T_i,T_i} +2\,({\cal T}_i^\bot)^\flat +(\Div H_i^\bot)\,g_j, B_j\>
 +\<-\Div{h_j}-{\cal K}_j^\flat -({H}_j^\bot)^\flat\otimes({H}_j^\bot)^\flat \\
 && +\,\frac12\Upsilon_{h_j^\bot,h_j^\bot}
 +\frac12\Upsilon_{T_j^\bot,T_j^\bot}+2\,{\cal T}_j^\flat+(\Div H_j)\,g_j, B_j\>\Big)\,{\rm d}\vol_g,
\end{eqnarray}
where ${B}_j={\dt g_t}_{\,|\,t=0}$. Since
\[
 {\rm\frac{d}{dt}}\,J_{{\cal D},\Omega}({g}_t)_{|\,t=0}
 =\int_{\Omega}\dt Q(g_t)_{\,|\,t=0}\,{\rm d}\vol_g + \int_{\Omega} Q(g)\,\big(\dt\,{\rm d}\vol_{g_{t} {\,|\,t=0}}\big),
\]
by \eqref{E-Sc-var1a} and \eqref{E-dotvolg}, we have
\begin{eqnarray}\label{E-varJh-init2}
 \nonumber
 && {\rm\frac{d}{dt}}\,J_{{\cal D},\Omega}({g}_t)_{|\,t=0}
 = \int_{\Omega}\Big[\<-\Div{h_j}-{\cal K}_j^\flat-({H}_j^\bot)^\flat\otimes({H}_j^\bot)^\flat+\frac12\,\Upsilon_{h_j^\bot,h_j^\bot}+\frac12\,\Upsilon_{T_j^\bot,T_j^\bot}+2\,{\cal T}_j^\flat \\
 \nonumber
 &&\qquad +\sum\nolimits_{\,i\ne j}\big(-\Div{h_i^\bot}-({\cal K}_i^\bot)^\flat
 -{H}_i^\flat\otimes{H}_i^\flat+\frac12\,\Upsilon_{h_i,h_i}+\frac12\,\Upsilon_{T_i,T_i}+2\,({\cal T}_i^\bot)^\flat\big)\\
 && \qquad + \big(\,{\rm S}_{\,{\mD}_1,\ldots,\mD_k} -\Div(H_j +\sum\nolimits_{\,i\ne j} {H}_i^\bot)\big)\,g_j,\ B_j\big\>\Big]\,{\rm d}\vol_g.
\end{eqnarray}
If $g$ is critical for $J_{{\cal D},\Omega}$ with respect to ${\cal D}_j$-variations of $g$,
then the integral in \eqref{E-varJh-init2} is zero for any symmetric $(0,2)$-tensor $B_j$.
This yields the ${\cal D}_j$-component of Euler-Lagrange equation
\begin{eqnarray}\label{ElmixDD}
\nonumber
 && \Div{h_j}+{\cal K}_j^\flat+({H}_j^\bot)^\flat\otimes({H}_j^\bot)^\flat
  -\,\frac12\Upsilon_{h_j^\bot,h_j^\bot} -\frac12\Upsilon_{T_j^\bot,T_j^\bot} -2\,{\cal T}_j^\flat \\
\nonumber
 &&\hskip-6mm +\!\sum\nolimits_{\,i\ne j}
 \big(\Div{h_i^\bot}|_{\mD_j}+(P_j{\cal K}_i^\bot)^\flat{+}(P_j{H}_i)^\flat\otimes(P_j{H}_i)^\flat-\frac12\,\Upsilon_{P_jh_i, P_jh_i}
 {-}\frac12\,\Upsilon_{P_jT_i, P_jT_i} -2\,(P_j{\cal T}_i^\bot)^\flat\big) \\
 && =\big(\,{\rm S}_{\,{\mD}_1,\ldots,\mD_k} -\Div(H_j +\sum\nolimits_{\,i\ne j} {H}_i^\bot)\big)\,g_j .
\end{eqnarray}
For volume-preserving ${\cal D}_j$-variations, the Euler-Lagrange equation
(of the geometrical part of \eqref{Eq-Smix} with respect to volume-preserving adapted variations)
will be \eqref{ElmixDDvp} instead of \eqref{ElmixDD}. 
\end{proof}

\begin{remark}\rm
Using the partial Ricci tensor \eqref{E-Rictop2} and replacing $\Div{h_j}$ and ${\Div}\,h_i^\bot$ for $i\ne j$ in \eqref{ElmixDDvp}
according to \eqref{E-genRicN}, we can rewrite \eqref{ElmixDDvp} as
\begin{eqnarray}\label{E-main-0i}
\nonumber
 &&\hskip-6mm r_{j}-\<h_j,\,H_j\>+{\cal A}_j^\flat -{\cal T}_j^\flat +\,\Psi_j^\bot -{\rm Def}_{\mD_j^\bot}\,H_j^\bot
 +{\cal K}_j^\flat+({H}_j^\bot)^\flat\otimes({H}_j^\bot)^\flat-\frac12\Upsilon_{h_j^\bot,h_j^\bot}-\frac12\Upsilon_{T_j^\bot,T_j^\bot}\\
\nonumber
 &&+\sum\nolimits_{\,i\ne j}
 \big(r^\bot_{i}|_{\mD_j}-\<h_i^\bot|_{\mD_j},\,H_i^\bot\>+(P_j{\cal A}^\bot_i)^\flat
 -(P_j{\cal T}^\bot_i)^\flat+\Psi_i|_{\mD_j}-{\rm Def}_{\mD_j}\,H_i +(P_j{\cal K}_i^\bot)^\flat\\
\nonumber
 && +\,(P_j{H}_i)^\flat\otimes(P_j{H}_i)^\flat -\frac12\Upsilon_{P_jh_i,P_jh_i}-\frac12\Upsilon_{P_jT_i,P_jT_i}  \big) \\
 && =\big(\,{\rm S}_{\,{\mD}_1,\ldots,\mD_k} -\Div(H_j +\sum\nolimits_{\,i\ne j}{H}_i^\bot) +\lambda_j \big)\,g_j ,
 \quad j=1,\ldots,k.
\end{eqnarray}
\end{remark}

\begin{example}
\rm
A pair $({\cal D}_i,{\cal D}_j)$ with $i\ne j$ of distributions on a Riemannian almost $k$-product manifold
$(M,g;{\cal D}_1,\ldots,{\cal D}_k)$ is called
\textit{mixed integrable}, if $T_{i,j}(X,Y)=0$ for all $X\in{\cal D}_i$ and $Y\in{\cal D}_j$, see \cite{r-IF-k}.
 Let $(M,g;{\cal D}_1,\ldots,{\cal D}_k)$ with $k>2$ has integrable distributions
${\cal D}_1,\ldots,{\cal D}_k$ and each pair $({\cal D}_i,{\cal D}_j)$ is mixed integrable.
Then $T_{l}^\bot(X,Y)=0$ for all $l\le k$ and $X\in{\cal D}_{i},\,Y\in{\cal D}_{j}$ with $i\ne j$, see \cite[Lemma~2]{r-IF-k}.
In this case, \eqref{E-main-0i} reads~as
\begin{eqnarray*}
 && r_{j}-\<h_j,\,H_j\>+{\cal A}_j^\flat +\Psi_j^\bot -\frac12\Upsilon_{h_j^\bot,h_j^\bot}
 +({H}_j^\bot)^\flat\otimes({H}_j^\bot)^\flat -{\rm Def}_{\mD_j^\bot}\,H_j^\bot \\
 &&+\sum\nolimits_{\,i\ne j} \big(r^\bot_{i}|_{\mD_j}-\<h_i^\bot|_{\mD_j},\,H_i^\bot\>+(P_j{\cal A}^\bot_i)^\flat
 +\Psi_i|_{\mD_j}-{\rm Def}_{\mD_j}\,H_i +(P_j{H}_i)^\flat\otimes(P_j{H}_i)^\flat \\
 && -\,\frac12\Upsilon_{P_jh_i,P_jh_i}\big)
 =\big({\rm S}_{\,{\mD}_1,\ldots,\mD_k} -\Div(H_j +\sum\nolimits_{\,i\ne j} {H}_i^\bot) +\lambda_j\big)\,g_j , \ \ j=1,\ldots,k.
\end{eqnarray*}
\end{example}

\begin{definition}\label{D-Ric-Dk}\rm
The Ricci type symmetric $(0,2)$-tensor $\Ric_{\,\mD}$ in \eqref{E-gravity} is defined by its restrictions
$\Ric_{\,\mD\,|\,\mD_j\times\mD_j}$ on $k$ subbundles $\mD_j\times\mD_j$ of $TM\times TM$,
\begin{eqnarray}\label{E-main-0ij-k}
\nonumber
 &&\hskip-4mm \Ric_{\,\mD\,|\,\mD_j\times\mD_j} = r_{j} -\<h_j,\,H_j\> +{\cal A}_j^\flat -{\cal T}_j^\flat +\Psi_j^\bot
 -{\rm Def}_{\mD_j^\bot}\,H_j^\bot +{\cal K}_j^\flat +({H}_j^\bot)^\flat\otimes({H}_j^\bot)^\flat \\
 \nonumber
 &&\hskip-5mm -\,\frac12\Upsilon_{h_j^\bot,h_j^\bot}
 -\frac12\Upsilon_{T_j^\bot,T_j^\bot}
 +\sum\nolimits_{\,i\ne j} \big(r^\bot_{i}|_{\mD_j}-\<h_i^\bot|_{\mD_j},\,H_i^\bot\>+(P_j{\cal A}^\bot_i)^\flat
  -(P_j{\cal T}^\bot_i)^\flat +\Psi_i|_{\mD_j} \\
 &&\hskip-5mm -\,{\rm Def}_{\mD_j}\,H_i
 +(P_j{\cal K}_i^\bot)^\flat +(P_j{H}_i)^\flat\otimes(P_j{H}_i)^\flat -\frac12\,\Upsilon_{P_jh_i,P_jh_i}
  -\frac12\,\Upsilon_{P_jT_i,P_jT_i} \big) +\mu_j \,g_j .
\end{eqnarray}
In other words,
\[
 \Ric_{\,\mD\,|\,\mD_j\times\mD_j} = U_j + \mu_j\,g_j,
\]
where $U_j$ is the LHS of \eqref{ElmixDDvp}
and $(\mu_j)$ are uniquely determined so (see Theorem~\ref{P-Ric-D} below) that critical metrics satisfy Einstein type equation \eqref{E-gravity}.
\end{definition}

\begin{theorem}\label{P-Ric-D}
A metric $g\in{\rm Riem}(M,\mD_1,\ldots\mD_k)$ is critical for the geometrical part of \eqref{Eq-Smix}
(i.e., $\Lambda=0=\Theta$) with respect to adapted variations if and only if
$g$ satisfies Einstein type equation \eqref{E-gravity}, where the tensor
$\Ric_{\,\mD}$ is given in Definition~\ref{D-Ric-Dk} with some (uniquely defined) $\mu_i\in\RR$.
\end{theorem}

\begin{proof} The Euler-Lagrange equations \eqref{ElmixDDvp} consist of ${\mD}_j\times\mD_j$-components.
Thus, for the geometrical part of \eqref{Eq-Smix} we obtain \eqref{E-main-0ij-k}.
If $n=2$ (and $k=2$), then we take $\mu_1=\mu_2=0$, see \cite{r2018}.
Assume that $n>2$. Substituting \eqref{E-main-0ij-k} with arbitrary $(\mu_1,\ldots,\mu_k)$ into \eqref{E-gravity} along $\mD_j$,
we conclude that if the Euler Lagrange equations $U_j=b_j\,g_j\ (1\le j\le k)$ hold, then
\[
 \Ric_{\,{\mD}}-(1/2)\,{\cal S}_{\,\mD}\cdot g = 0,
\]
see \eqref{E-gravity} with $\Lambda=0=\Theta$, if and only if $(\mu_j)$ satisfy the following
linear~system:
\begin{equation}\label{E-mu-system}
  \sum\nolimits_{\,i} n_i\,\mu_i -2\,\mu_j = a_j,\quad j=1,\ldots,k,
\end{equation}
with coefficients $a_j=2\,b_j -\tr\,\sum_i U_i$.
The matrix of \eqref{E-mu-system} is
\[
 A=
 \left(
   \begin{array}{ccccc}
     n_1-2 & n_2 & \ldots& n_{k-1} & n_k \\
     n_1 & n_2-2 & \ldots& n_{k-1} & n_k \\
     \ldots & \ldots & \ldots& \ldots & \ldots \\
     n_1 & n_2 & \ldots & n_{k-1} & n_k-2 \\
   \end{array}
 \right) .
\]
Its determinant: $\det A=2^{k-1}(2-n)$ is negative when $n>2$. Hence, the system \eqref{E-mu-system} has a unique solution
$(\mu_1,\ldots,\mu_k)$.
It is given by $ \mu_{i}=-\frac1{2n-4}\,\big(\sum\nolimits_{\,j}\,(a_{i}-a_{j})\,n_{j}-2\,a_{i}\big)$.
\end{proof}

\begin{example}[see \cite{r2018}]\label{D-Ric-D}\rm
The symmetric Ricci type tensor $\Ric_{\,\mD}$ in \eqref{E-gravity} with $k=2$, is defined by its restrictions on two
subbundles of $TM\times TM$,
\begin{eqnarray}\label{E-main-0ij}
\nonumber
 && \Ric_{\,\mD\,|\,\mD^\bot\times\mD^\bot} = {r}
 -\<h^\bot,\,H^\bot\>+({\cal A}^\bot)^{\,\flat}-({\cal T}^\bot)^{\,\flat}+\Psi-{\rm Def}_{\mD}\,H+({\cal K}^\bot)^{\,\flat} \\
\nonumber
 && \hskip10mm +\,H^\flat\otimes H^\flat -\frac{1}{2}\,\Upsilon_{\,h,h} -\frac12\,\Upsilon_{\,T,T}+\mu_1\,g^\perp, \\
\nonumber
 &&  \Ric_{\,\mD|\,\mD\times\mD} = {r}^\bot-\<h,\,H\>+{\cal A}^\flat-{\cal T}^\flat
  +\Psi^\bot -{\rm Def}_{{\mD}^\bot}\,H^\bot +{\cal K}^\flat \\
 &&  \hskip10mm +\,(H^\bot)^\flat\otimes (H^\bot)^\flat -\frac{1}{2}\,\Upsilon_{\,h^\bot, h^\bot} -\frac12\,\Upsilon_{\,T^\bot, T^\bot} +\mu_2\,g^\top ,
\end{eqnarray}
where $\mu_1=-\frac{n_1-1}{n-2}\,\Div(H^\bot-{H})$ and $\mu_2=\frac{n_2-1}{n-2}\,\Div(H^\bot-{H})$.
Here \eqref{E-main-0ij}$_2$ is dual to \eqref{E-main-0ij}$_1$
with respect to interchanging distributions ${\mD}$ and ${\mD}^\bot$,
and their last terms vanish if $n_1=n_2=1$.
Also, $\,{\cal S}_{\,\mD} := \tr_g\Ric_{\,\mD} = {\rm S}_{\,\mD,\mD^\bot} + \frac{n_2-n_1}{n-2}\,\Div(H^\bot-{H})$.
\end{example}


\end{document}